\documentclass{article}

\usepackage[T1]{fontenc}
\usepackage{amssymb,amsfonts,amsmath,amsthm}
\usepackage{tikz}
\usetikzlibrary{cd}
\usepackage{hyperref}
\usepackage{url}

\usepackage{graphicx}
\usepackage[margin=2.5cm]{geometry}
\usepackage{fancyhdr}
\fancypagestyle{firststyle}
{
   \fancyhf{}
   
   \fancyfoot[C]{
     \footnotesize 
     \raisebox{-.8em}{\includegraphics[height=2em]{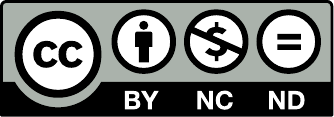}} \ \parbox{.7\textwidth}{
     \copyright 2017. This manuscript version is made available under the CC-BY-NC-ND 4.0 license \url{http://creativecommons.org/licenses/by-nc-nd/4.0/}}
   }
}

\newcommand{\Z}{{\mathbb Z}}
\newcommand{\Q}{{\mathbb Q}}
\newcommand{\R}{{\mathbb R}}
\newcommand{\C}{{\mathbb C}}

\newtheorem{theo}{Theorem}
\newtheorem{prop}{Proposition}
\newtheorem{lm}{Lemma}
\theoremstyle{definition}
\newtheorem{que}{Question}
\newtheorem{ex}{Example}
\theoremstyle{remark}
\newtheorem{rem}{Remark}

\DeclareMathOperator{\GL}{GL}

\date{}

\begin{document}

\title{A short note about diffuse Bieberbach groups}
\author{
Rafa{\l} Lutowski\thanks{Institute of Mathematics, University of Gda\'nsk, ul. Wita Stwosza 57, 80-952 Gda\'nsk, Poland}, \ Andrzej Szczepa\'nski\footnotemark[1]{} \ and 
Anna G\k{a}sior\thanks{Institute of Mathematics, Maria Curie-Sk{\l}odowska University, Pl. Marii Curie-Sk{\l}odowskiej 1, 20-031 Lublin, Poland}}


\maketitle
\thispagestyle{firststyle}

\begin{abstract}
We consider low dimensional diffuse Bieberbach groups. In particular we classify diffuse Bieberbach groups up to dimension 6. We also answer a question from \cite[page 887]{KRD} about the minimal dimension of a non-diffuse Bieberbach group which does not contain the three-dimensional Hantzsche-Wendt group.
\end{abstract}

\renewcommand{\thefootnote}{}
\footnotetext{\emph{E-mail addresses:} rafal.lutowski@mat.ug.edu.pl, aszczepa@mat.ug.edu.pl, anna.gasior@poczta.umcs.lublin.pl }
\renewcommand{\thefootnote}{\arabic{footnote}}

\section{Introduction}
The class of diffuse groups was introduced by B. Bowditch in \cite{Bow}.
By definition a group $\Gamma$ is \emph{diffuse}, if every finite non-empty subset $A\subset\Gamma$ has an
extremal point, i.e. an element $a\in A$ such that for any $g\in\Gamma\setminus\{1\}$ either $ga$ or $g^{-1}a$
is not in $A.$ Equivalently (see \cite{KRD}) a group $\Gamma$ is diffuse if it does not
contain a non-empty finite set without extremal points. 

The interest in diffuse groups follows from Bowditch's observation that they have the unique product property
\footnote{The group $\Gamma$ is said to have the unique product property if for every two finite non-empty subsets $A,B\subset\Gamma$
there is an element in the product $x\in A\dot B$ which can be written uniquely in the form $x = ab$ with $a\in A$ and $b\in B.$}.
Originally unique products were introduced in the study of group rings of discrete, torsion-free groups.
More precisely, it is easily seen that if a group $\Gamma$ has the unique product property, then it satisfies
Kaplansky's unit conjecture. In simple terms this means that the units in the group ring $\C[\Gamma]$ are all trivial,
i.e. of the form $\lambda g$ with $\lambda\in\C^{\ast}$ and $g\in\Gamma.$ 
For more information about these objects we refer the reader to
\cite{BLR}, \cite[Chapter 10]{L} and  \cite{KRD}. 
In part 3 of \cite{KRD} the authors prove that any torsion-free crystallographic group (Bieberbach group)
with trivial center is not diffuse. By definition a crystallographic group is a discrete and cocompact subgroup of the
group $O(n)\ltimes\R^n$ of isometries of the Euclidean space $\R^n.$ 
From Bieberbach's theorem (see \cite{S}) the normal subgroup $T$ of all translations of any crystallographic group
$\Gamma$ is a free abelian group of finite rank and the quotient group (holonomy group)  $\Gamma/T = G$ is finite.

In  \cite[Theorem 3.5]{KRD} it is proved that for a finite group $G$:

\begin{enumerate}
\item If $G$ is not solvable then any Bieberbach group with holonomy group isomorphic to $G$ \emph{is not} diffuse.
\item If every Sylow subgroup of $G$ is cyclic then any Bieberbach group with holonomy group isomorphic to $G$ \emph{is} diffuse.
\item If $G$ is solvable and has a non-cyclic Sylow subgroup then there are examples of Bieberbach groups with holonomy group isomorphic to $G$ which \emph{are} and examples which \emph{are not} diffuse.
\end{enumerate}

Using the above the authors of \cite{KRD} classify non-diffuse Bieberbach groups in dimensions $\leq 4.$
One of the most important non-diffuse groups is the 3-di\-men\-sio\-nal Hantzsche-Wendt group, denoted in \cite{P} by $\Delta_P$. For the following presentation
\[\Delta_{P} = \langle x,y\mid x^{-1}y^{2}x = y^{-2}, y^{-1}x^{2}y = x^{-2}\rangle\] 
the maximal abelian normal subgroup is generated by $x^2,y^2$ and $(xy)^2$ (see \cite[page 154]{H}).
At the end of part 3.4 of \cite{KRD} the authors ask the following question.

\begin{que}\label{HW}
What is the smallest dimension $d_0$ of a non-diffuse Bieberbach group which does not contain $\Delta_{P}$ ?
\end{que}

The answer for the above question was the main motivation for us. In fact we prove, in the next section, that $d_{0} = 5.$
Moreover, we extend the results of part 3.4 of \cite{KRD} and with support of computer, we present the classification of all Bieberbach groups in dimension 
$d\leq 6$ which are (non)diffuse. 

\section{(Non)diffuse Bieberbach groups in dimension $\leq 6.$}

We use the computer system CARAT \cite{Car} to list all Bieberbach groups of dimension $\leq 6.$

Our main tools are the following observations:
\begin{enumerate}
\item The property of being diffuse is inherited by subgroups
(see \cite[page 815]{Bow}).
\item If $\Gamma$ is a torsion-free group, $N\lhd\Gamma$ such that $N$ and $\Gamma/N$ are both diffuse then $\Gamma$ is diffuse (see \cite[Theorem 1.2 (1)]{Bow}).
\end{enumerate}

Now let $\Gamma$ be a Bieberbach group of dimension less than or equal to $6$. By the first Betti number $\beta_1(\Gamma)$ we mean the rank of the abelianization $\Gamma/[\Gamma,\Gamma]$. Note that we are only interested in the case when $\beta_1(\Gamma) > 0$ (see \cite[Lemma 3.4]{KRD}). Using a method of E. Calabi \cite[Propostions 3.1 and 4.1]{S},
we get an epimorphism 
\begin{equation}\label{calabi}
f:\Gamma\to\Z^{k}, \text{ where } k = \beta_1(\Gamma). 
\end{equation}
From the assumptions $\ker f$ is a Bieberbach
group of dimension $< 6.$ Since $\Z^{k}$ is a diffuse group our problem is reduced to the question about the group $\ker f.$

\begin{rem}
Up to conjugation in $\GL(n+1,\R)$, $\Gamma$ is a subgroup $\GL(n,\Z) \ltimes \Z^n \subset \GL(n+1,\Q)$, i.e. it is a group of matrices of the form
\[
\begin{bmatrix}
A & a\\
0 & 1
\end{bmatrix},
\]
where $A \in \GL(n,\Z), a \in \Q^n$. If $p \colon \Gamma \to \GL(n,\Z)$ is a homomorphism which takes the linear part of every element of $\Gamma$
\[
p\left(\begin{bmatrix}A&a\\0&1\end{bmatrix}\right) = A \text{ for every } \begin{bmatrix}A&a\\0&1\end{bmatrix} \in \Gamma,
\]
then there is an isomorphism $\rho \colon G \to p(\Gamma) \subset \GL(n,\Z)$. It is known that that the rank of the center of a Bieberbach group equals the first Betti number (see \cite[Proposition 1.4]{HS}). By \cite[Lemma 5.2]{S}, the number of trivial constituents of the representation $\rho$ is equal to $k$. Hence without lose of generality we can assume that the matrices in $\Gamma$ are of the form
\[
\begin{bmatrix}
A & B & a\\
0 & I & b\\
0 & 0 & 1
\end{bmatrix}
\]
where $A \in \GL(n-k,\Z)$, $I$ is the identity matrix of degree $k$, $B$ is an integral matrix of dimension $n-k \times k$, $a \in \Q^{n-k}$ and $b \in \Q^k$. Then $f$ may be defined by
\[
f\left(
\begin{bmatrix}
A & B & a\\
0 & I & b\\
0 & 0 & 1
\end{bmatrix}
\right)
= b
\]
and one can easily see that the map $F \colon \ker f \to \GL(n-k+1,\Q)$ given by
\[
F\left(
\begin{bmatrix}
A & B & a\\
0 & I & 0\\
0 & 0 & 1
\end{bmatrix}
\right)
=
\begin{bmatrix}
A  & a\\
0  & 1
\end{bmatrix}
\]
is a monomorphism and hence its image is a Bieberbach group of rank $n-k$.
\end{rem}

Now if $\Gamma$ has rank $4$ we know that the only non-diffuse Bieberbach group of dimension less than or equal to $3$ is $\Delta_{P}.$
Using the above facts we obtain 17 non-diffuse groups. Note that the list from \cite[section 3.4]{KRD} consists of 16 groups. 
The following example presents the one which is not in \cite{KRD} and illustrates computations given in the above remark.

\begin{ex}
Let $\Gamma$ be a crystallographic group denoted by "05/01/06/006" in \cite{BBNWZ} as a subgroup of $\operatorname{GL}(5,\R)$. Its non-lattice generators are as follows
\[
A =
\begin{bmatrix}
  0 &  -1 &  1 &   0 &  1/2\\
 -1 &   0 &  1 &   0 &    0\\
  0 &   0 &  1 &   0 &  1/2\\
  0 &   0 &  0 &  -1 &  1/2\\
  0 &   0 &  0 &   0 &    1\\
\end{bmatrix}
\text{ and }
B =
\begin{bmatrix}
  0 &  1 &  -1 &   0 &  1/2\\
  0 &  1 &   0 &   0 &  1/2\\
 -1 &  1 &   0 &   0 &    0\\
  0 &  0 &   0 &  -1 &    0\\
  0 &  0 &   0 &   0 &    1\\
\end{bmatrix}.
\]
Conjugating the above matrices by
\[
Q=
\begin{bmatrix}
 1 &  1 &  0 &  0 &  0\\
 0 &  1 &  0 &  0 &  0\\
 1 &  0 &  0 &  1 &  0\\
 0 &  0 &  1 &  0 &  0\\
 0 &  0 &  0 &  0 &  1\\
\end{bmatrix}
\in \operatorname{GL}(5,\mathbb{Z})
\]
one gets
\[
A^Q =
\begin{bmatrix}
 1 &   0 &   0 &  0 &  1/2\\
 0 &  -1 &   0 &  1 &    0\\
 0 &   0 &  -1 &  0 &  1/2\\
 0 &   0 &   0 &  1 &    0\\
 0 &   0 &   0 &  0 &    1\\
\end{bmatrix}
\text{ and }
B^Q =
\begin{bmatrix}
 -1 &  0 &   0 &  -1 &    0\\
  0 &  1 &   0 &   0 &  1/2\\
  0 &  0 &  -1 &   0 &    0\\
  0 &  0 &   0 &   1 &    0\\
  0 &  0 &   0 &   0 &    1\\
\end{bmatrix}.
\]
Now its easy to see that the rank of the center of $\Gamma$ equals $1$ and the kernel of 
the epimorphism $\Gamma \to \mathbb{Z}$ is isomorphic to a 3-dimensional Bieberbach group $\Gamma'$ with the following non-lattice generators:
\[
A' =
\begin{bmatrix}
 1 &   0 &   0 & 1/2\\
 0 &  -1 &   0 &   0\\
 0 &   0 &  -1 & 1/2\\
 0 &   0 &   0 &   1\\
\end{bmatrix}
\text{ and }
B' =
\begin{bmatrix}
 -1 &  0 &   0 &   0\\
  0 &  1 &   0 & 1/2\\
  0 &  0 &  -1 &   0\\
  0 &  0 &   0 &   1\\
\end{bmatrix}
\]
Clearly the center of $\Gamma'$ is trivial, hence it is isomorphic to the group $\Delta_{P}.$
\end{ex}
Now we formulate our main result.

\begin{theo}
The following table summarizes the number of diffuse and non-diffuse Bieberbach groups of dimension $\leq 6$.

\begin{center}
\normalfont
\begin{tabular}{r|r|r|r}
Dimension&Total&Non-diffuse&Diffuse\\  \hline
   1 &     1 &    0 &             1\\
   2 &     2 &    0 &             2\\
   3 &    10 &    1 &             9\\
   4 &    74 &   17 &            57\\
   5 &  1060 &  352 &           708\\
   6 & 38746 &19256 &         19490\\
\end{tabular}
\end{center}
\end{theo}
\noindent
\begin{proof}
If a group has a trivial center then it is not diffuse. In other case we use the Calabi (\ref{calabi}) method and induction.
A complete list of groups was obtained using computer algebra system GAP \cite{GAP} and is available here \cite{RL}.
\end{proof}

Before we answer Question~\ref{HW} from the introduction, let us formulate the following lemma:
\begin{lm}
\label{lm:hwpres}
Let $\alpha, \beta$ be any generators of the group $\Delta_{P}$. Let $\gamma = \alpha \beta, a=\alpha^2, b=\beta^2, c=\gamma^2$. Then the following relations hold:
\begin{equation}
\label{eq:hwrels}
\begin{array}{llll}
[a,b]=1 &                   & a^\beta = a^{-1} & a^\gamma = a^{-1} \cr
[a,c]=1 & b^\alpha = b^{-1} &                  & b^\gamma = b^{-1} \cr
[b,c]=1 & c^\alpha = c^{-1} & c^\beta = c^{-1} &                   \cr 
\end{array}
\end{equation}
where $x^y := y^{-1}xy$ denotes the conjugation of $x$ by $y$.
\end{lm}
The proof of the above lemma is omitted. Just note that the relations are easily checked if consider the following representation of $\Delta_P$ as a matrix group
\[
\left\langle
x=
\begin{bmatrix}
 1 &  0 &  0 & 1/2\\
 0 & -1 &  0 & 1/2\\
 0 &  0 & -1 &   0\\
 0 &  0 &  0 &   1
\end{bmatrix},
y=
\begin{bmatrix}
-1 &  0 &  0 &   0\\
 0 &  1 &  0 & 1/2\\
 0 &  0 & -1 & 1/2\\
 0 &  0 &  0 &   1
\end{bmatrix}
\right\rangle
\subset \GL(4,\Q).
\]

\begin{prop}
There exists an example of a five dimensional non-diffuse Bieberbach group which does not contain any subgroup isomorphic to $\Delta_{P}$.
\end{prop} 
\begin{proof}
Let $\Gamma$ be the Bieberbach group enumerated in CARAT as "min.88.1.1.15". It generated by the elements $\gamma_1, \gamma_2,l_1,\ldots,l_5$ where
\[
\gamma_1 =
\begin{bmatrix}
 0 &  1 &  0 &   0 &   0 &    0\\
 1 &  0 &  0 &   0 &   0 &    0\\
 0 &  0 &  1 &   0 &   0 &  1/2\\
 0 &  0 &  0 &  -1 &   0 &  1/4\\
 0 &  0 &  0 &   0 &  -1 &    0\\
 0 &  0 &  0 &   0 &   0 &    1\\
\end{bmatrix}
\text{ and }
\gamma_2 =
\begin{bmatrix}
 1 &   0 &   0 &   0 &  0 &  1/2\\
 0 &  -1 &   0 &   0 &  0 &    0\\
 0 &   0 &  -1 &   0 &  0 &    0\\
 0 &   0 &   0 &  -1 &  0 &    0\\
 0 &   0 &   0 &   0 &  1 &  1/2\\
 0 &   0 &   0 &   0 &  0 &    1\\
\end{bmatrix}
\]
and $l_1,\ldots,l_5$ generate the lattice $L$ of $\Gamma$:
\[
l_i := 
\begin{bmatrix}
I_5 & e_i\\
0   & 1
\end{bmatrix}
\]
where $e_i$ is the $i$-th column of the identity matrix $I_5$. $\Gamma$ fits into the following short exact sequence
\[
1 \longrightarrow L \longrightarrow \Gamma \stackrel{\pi}{\longrightarrow} D_8 \longrightarrow 1
\]
where $\pi$ takes the linear part of every element of $\Gamma$:
\[
\begin{bmatrix}
A & a\\
0 & 1
\end{bmatrix}
\mapsto A
\]
and the image $D_8$ of $\pi$ is the dihedral group of order 8.

Now assume that $\Gamma'$ is a subgroup of $\Gamma$ isomorphic to $\Delta_P$. Let $T$ be its maximal normal abelian subgroup. Then $T$ is free abelian group of rank $3$ and $\Gamma'$ fits into the following short exact sequence
\[
1 \longrightarrow T \longrightarrow \Gamma' \longrightarrow C_2^2 \longrightarrow 1,
\]
where $C_m$ is a cyclic group of order $m$. Consider the following commutative diagram
\[
\begin{tikzcd}[ampersand replacement=\&]
1 \arrow{r} \& T \cap L \arrow{r}\arrow{d} \& T \arrow{r} \arrow{d} \& H = \pi(T) \arrow{r} \arrow[d,dash,shift right=.1em] \& 1\\
1 \arrow{r} \& L \arrow{r} \& \pi^{-1}(H) \arrow{r} \& H  \arrow{r} \arrow[u,dash,shift right=.1em] \& 1\\
\end{tikzcd}
\]
We get that $H$ must be an abelian subgroup of $D_8 = \pi(\Gamma)$ and $T\cap L$ is a free abelian group of rank $3$ which lies in the center of $\pi^{-1}(H) \subset \Gamma$. Now if $H$ is isomorphic to either to $C_4$ or $C_2^2$ then the center of $\pi^{-1}(H)$ is of rank at most $2$. Hence $H$ must be the trivial group or the cyclic group of order $2$. Note that as $\Gamma' \cap L$ is a normal abelian subgroup of $\Gamma'$ it must be a subgroup of $T$:
\[
T \cap L \subset \Gamma' \cap L \subset T \cap L,
\]
hence $T \cap L = \Gamma' \cap L$.
We get the following commutative diagram with exact rows and columns
\[
\begin{tikzcd}[ampersand replacement=\&]
 \& 1 \arrow{d} \& 1 \arrow{d} \& 1 \arrow{d}\\
1 \arrow{r} \& T \cap L \arrow{r}\arrow{d} \& T \arrow{r}\arrow{d}  \& H \arrow{r}\arrow{d} \& 1\\
1 \arrow{r} \& \Gamma' \cap L \arrow{r}\arrow{d} \& \Gamma' \arrow{r}\arrow{d}  \& G \arrow{r}\arrow{d} \& 1\\
1 \arrow{r} \& 1 \arrow{r}\arrow{d} \& C_2^2 \arrow{r}\arrow{d}  \& C_2^2 \arrow{r} \arrow{d} \& 1\\
 \& 1 \& 1 \& 1\\
\end{tikzcd}
\]
where $G = \pi(\Gamma')$. 
Consider two cases:
\begin{enumerate}
\item $H$ is trivial. In this case $G$ is one of the two subgroups of $D_8$ isomorphic to $C_2^2$. Since the arguments for both subgroups are similar, we present only one of them. Namely, let
\[
G = \left\langle
\operatorname{diag}(1,-1,-1,-1,1),
\operatorname{diag}(-1,-1,1,1,1).
\right\rangle
\] 
In this case $\Gamma'$ is generated by the matrices of the form
\[
\alpha = 
\begin{bmatrix}1 & 0 & 0 & 0 & 0 & {{x}_{1}}-\frac{1}{2}\\
0 & -1 & 0 & 0 & 0 & {{x}_{2}}\\
0 & 0 & -1 & 0 & 0 & {{x}_{3}}\\
0 & 0 & 0 & -1 & 0 & {{x}_{4}}\\
0 & 0 & 0 & 0 & 1 & {{x}_{5}}-\frac{1}{2}\\
0 & 0 & 0 & 0 & 0 & 1\end{bmatrix}
\text{ and }
\beta = 
\begin{bmatrix}-1 & 0 & 0 & 0 & 0 & {{y}_{1}}+\frac{1}{2}\\
0 & -1 & 0 & 0 & 0 & {{y}_{2}}-\frac{1}{2}\\
0 & 0 & 1 & 0 & 0 & {{y}_{3}}\\
0 & 0 & 0 & 1 & 0 & {{y}_{4}}+\frac{1}{2}\\
0 & 0 & 0 & 0 & 1 & {{y}_{5}}\\
0 & 0 & 0 & 0 & 0 & 1\end{bmatrix},
\]
where $x_i,y_i \in \Z$ for $i=1,\ldots,5$. If $c = (\alpha\beta)^2$ then by Lemma \ref{lm:hwpres} $c^\alpha = c^{-1}$, but
\[
c^\alpha - c^{-1} =
\begin{bmatrix}0 & 0 & 0 & 0 & 0 & 0\\
0 & 0 & 0 & 0 & 0 & 0\\
0 & 0 & 0 & 0 & 0 & 0\\
0 & 0 & 0 & 0 & 0 & 0\\
0 & 0 & 0 & 0 & 0 & 4{{y}_{5}}+4{{x}_{5}}-2\\
0 & 0 & 0 & 0 & 0 & 0\end{bmatrix}.
\]
Obviously solutions of the equation $4{{y}_{5}}+4{{x}_{5}}-2=0$ are never integral and we get a contradiction.

\item $H$ is of order $2$. Then $G = D_8$ and $H$ is the center of $G$. The generators $\alpha,\beta$ of $\Gamma'$ lie in the cosets $\gamma_1\gamma_2L$ and $\gamma_2L$, hence
\[
\alpha = 
\begin{bmatrix}0 & -1 & 0 & 0 & 0 & {{x}_{1}}\\
1 & 0 & 0 & 0 & 0 & {{x}_{2}}-\frac{1}{2}\\
0 & 0 & -1 & 0 & 0 & {{x}_{3}}-\frac{1}{2}\\
0 & 0 & 0 & 1 & 0 & {{x}_{4}}+\frac{1}{4}\\
0 & 0 & 0 & 0 & -1 & {{x}_{5}}+\frac{1}{2}\\
0 & 0 & 0 & 0 & 0 & 1\end{bmatrix}
\text{ and }
\beta = 
\begin{bmatrix}1 & 0 & 0 & 0 & 0 & {{y}_{1}}-\frac{1}{2}\\
0 & -1 & 0 & 0 & 0 & {{y}_{2}}\\
0 & 0 & -1 & 0 & 0 & {{y}_{3}}\\
0 & 0 & 0 & -1 & 0 & {{y}_{4}}\\
0 & 0 & 0 & 0 & 1 & {{y}_{5}}-\frac{1}{2}\\
0 & 0 & 0 & 0 & 0 & 1\end{bmatrix}
\]
where $x_i,y_i \in \Z$ for $i=1,\ldots,5$, as before. Setting $a=\alpha^2, b=\beta^2$ we get
\[
ab-ba =
\begin{bmatrix}0 & 0 & 0 & 0 & 0 & 2-4{{y}_{1}}\\
0 & 0 & 0 & 0 & 0 & 0\\
0 & 0 & 0 & 0 & 0 & 0\\
0 & 0 & 0 & 0 & 0 & 0\\
0 & 0 & 0 & 0 & 0 & 0\\
0 & 0 & 0 & 0 & 0 & 0\end{bmatrix}
\]
and again the equation $ 2-4{{y}_{1}} = 0$ does not have an integral solution.
\end{enumerate}
The above considerations show that $\Gamma$ does not have a subgroup which is isomorphic to $\Delta_P$.
\end{proof}

\section*{Acknowledgements}
This work was supported by the Polish National Science Center grant 2013/09/B/ST1/04125.

\end{document}